\documentclass{amsart}
\usepackage{amscd,amssymb,amsmath,amsthm,amsfonts,stmaryrd}
\usepackage[all]{xy}
\usepackage{nccmath}
\usepackage{graphicx}

\numberwithin{equation}{section}

\newtheorem{prop}{Proposition}
\newtheorem{thm}{Theorem}
\newtheorem*{theo}{Theorem}
\newtheorem{lem}{Lemma}
\newtheorem{cor}{Corollary}
\newcommand\Th{\mathcal{TH}}
\newcommand\Thc{\mathcal{TH}^{1}}
\newcommand\Tho{\widetilde{\mathcal{TH}}}

\newcommand\hoc{\widetilde{\mathcal{H}}^{1}}
\newcommand\ho{\widetilde{\mathcal{H}}}
\newcommand\h{\mathcal{H}}
\newcommand\hc{\mathcal{H}^{1}}

\newcommand\R{\mathbb{R}}
\newcommand\C{\mathbb{C}}
\newcommand\N{\mathbb{N}}
\newcommand\Z{\mathbb{Z}}

\newcommand\Hom{\text{Hom}}

\newtheorem{rem}{Remark}

\newcommand\n{\textsc{n}}

\author{Simon BARAZER }

\title{Average diffusion rate of Ehrenfest Wind-tree billiards}

\begin{document}

\maketitle

\begin{abstract}
One of the versions of the wind-tree model of Boltzmann gas, suggested by Paul and
Tatiana Ehrenfest more than a century ago, can be seen as a billiard
in the plane endowed with $\mathbb{Z}\oplus\mathbb{Z}$-periodic rectangular
obstacles. In the breakthrough paper by
V.~Delecroix, P.~Hubert and S.~Lelievre the authors proved,
that the diffusion
rate of trajectories in such a billiard is equal to $\frac{2}{3}$,
that is the \textit{maximal} distance from the origin
achieved by a point of
a typical trajectory on a segment of time $[0,t]$
grows roughly as $t^\frac{2}{3}$ for large $t$.
Here $\frac{2}{3}$ is the Lyapunov exponent of the
associated renormalizing dynamical system.

This pioneering result does
not tell, however, whether trajectories spend most of the time close
or far from the initial point. In the current paper, we prove
that the
\textit{average} distance from the origin grows
with the same rate $t^\frac{2}{3}$.
In plain terms, it means that
trajectories mostly stay
as far as possible from the initial point
(though, it is known that the wind-tree billiard is recurrent, so
trajectories occasionally pass close to the initial point).

More generally, fundamental rigidity results by A.~Eskin and
M.~Mirzakhani completed by certain genericity results by J.~Chaika
and A.~Eskin imply that the diffusion rate of almost all flat
geodesic rays on any $\mathbb{Z}^d$-cover of a closed translation
surface $S$ is given by certain Lyapunov exponent
of the
Kontsevich--Zorich cocycle on the
$\text{SL}_2(\mathbb{R})$-orbit closure of $S$. In this paper
we prove that in this most general setting, the
\textit{average} and \textit{maximum}
diffusion rates coincide.
\end{abstract}

 \begin{figure}[h!]
 \centering
 \includegraphics[width =6cm]{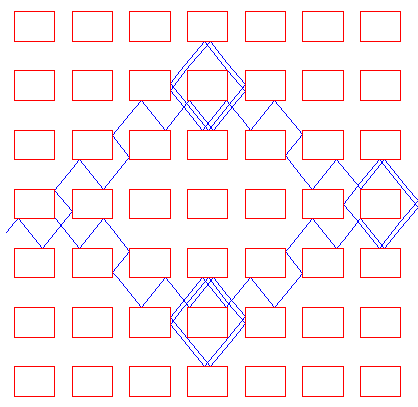}
 \caption{Trajectory in a Wind-Tree}
 \label{windtree1}
 \end{figure}

\newpage

\section{Introduction}
The wind-tree model is an infinite $\Z^2$-periodic billiard in the
plane $\R^2$, where identical rectangular obstacles are placed in
such way that the corresponding sides are aligned with the horizontal
and vertical axes and the centers of rectangles are located at the
lattice points, see figure~\ref{windtree1}). More formally, for every
couples of real parameters  $(a,b)\in (0,1)^2$ the wind-tree table
$T(a,b)$ is defined as
\begin{equation*}
T(a,b)=\R^2\setminus
    \bigsqcup_{(m,n)\in \Z^2} \left (m-\frac{a}{2},m+\frac{a}{2}\right ) \times \left (n-\frac{b}{2},n+\frac{b}{2}\right ).
\end{equation*}
It's assumed that reflections on the sides of
the rectangles are regular and the speed of all trajectories is
constant equal to one (figure \ref{windtree1}). A version of this
model was introduced by Tatiana and Paul~Erhenfest more than a
century ago, see~\cite{ehrenfest1990conceptual}. It have been studied
more recently by using the tools of translation surfaces theory.
Athur~Avila  and Pascal~Hubert proved in~\cite{avila2020recurrence}
that the trajectories in the
wind-tree are recurrent. Later Vincent~Delecroix, Pascal~Hubert and
Samuel~Leli\`evre proved in~\cite{delecroix2014diffusion}
that the diffusion rate of a typical
trajectory is equal to $\frac{2}{3}$. If $\gamma_t(x,\theta)$ is the
wind-tree flow in the direction $\theta$ starting from $x\in T(a,b)$,
they proved the following theorem.

\begin{theo}[Delcroix, Hubert, Lelievre, 2014]
For any $(a,b)\in (0,1)^2$,
for almost every $\theta$ (depending on $(a,b)$) and for all $x\in
T(a,b)$ having an infinite future orbit, the following relation holds:
\begin{equation*}
\limsup_{t\rightarrow +\infty}
    \frac{\log d(x,\gamma_t(x,\theta))}{\log t} = \frac{2}{3}\,.
\end{equation*}
\end{theo}

Anton~Zorich and Vincent~Delecroix generalized this result to a large
class of wind-tree billiards with more complicated obstacles
\cite{delecroix2020cries}. Charles~Fougeron studied
in~\cite{fougeron2018diffusion} orbit closures of flat surfaces
associated to even more general wind-tree billiards. In particular,
he identified the Lyapunov exponent of the Kontsevich--Zorich cocycle
that governs the associated diffusion rate. In a different direction,
Krzysztof Fraczek and Corinna Ulcigrai showed
in~\cite{frkaczek2014non} that the directional flow for almost all
wind-tree billiard is non-ergodic.

Note that the diffusion rate $\frac{2}{3}$ is defined in the original
paper~\cite{delecroix2014diffusion} of
Delecroix--Hubert--Lelièvre as the \textit{maximal} distance from
the origin achieved by a point of a typical trajectory on a large
segment of time. In the present paper, we study the diffusion
\textit{in average}. In particular, we prove the following theorem:

\begin{thm}
\label{SB1}
For any $(a,b)\in (0,1)^2$,
for almost every $\theta$ (depending on $(a,b)$) and for all $x\in
T(a,b)$ having an infinite future orbit, the following relation holds:
\begin{equation*}
\lim_{T\rightarrow +\infty}
\frac{\log \frac{1}{T}\int_0^T d(x,\gamma_t(x,\theta))dt}{\log T}
= \frac{2}{3}\,.
\end{equation*}
\end{thm}

The strategy of the proof is similar to the proof of the Theorem of
Delecroix--Hubert--Lelièvre presented above. The first step is
standard: following the Katok--Zemliakov construction
(\cite{zemlyakov1975topological}, see also~\cite{masur2002rational}), the authors of \cite{delecroix2014diffusion}
construct a translation surface $X_\infty(a,b)$ without boundary from
four copies of the wind-tree table $T(a,b)$. By construction, the
billiard flow in $T(a,b)$ is conjugate to the geodesic flow on the
translation surface $X_\infty(a,b)$. The fact that the wind-tree
table is $\Z^2$-periodic implies that the surface $X_\infty(a,b)$ is a
$\Z^2$-cover over a compact translation surface $X(a,b)$. In this way,
we reduce the problem to a particular case of the study of geodesics
on $\Z^d$-covers of translation surfaces.

We call $\tilde{S}\to S$ a \textit{lattice cover} if $\tilde{S}$ is
a connected $\Z^d$-cover of a compact translation surface $S$.
We will see that such a cover is defined by a
cohomology class $f$ in the first cohomology $H^1(S,\Z^d)$.
As in~\cite{delecroix2014diffusion},
the diffusion of a billiard trajectory can be
studied through the intersection pairing of $f$
with geodesic segments on $S$ completed to closed paths
in a natural way. We denote by
$\tilde{\phi}_t(x,\theta),\phi_t(x,\theta)$ the geodesic flows on
$\tilde{S}$ and $S$ respectively.
By $d(x,y)$ we denote the distance between points $x$ and $y$
on the translation surface $\tilde S$.
We prove the following theorem.

\begin{thm}
\label{SB}
Let $\tilde{S}\rightarrow S$ a lattice cover; let $f$ be the
associated cohomology class in $H^1(S,\Z^d)$. For every $x\in\tilde S$
and almost every $\theta$ we have
\begin{equation*}
\lim_{T\rightarrow +\infty}
\frac{\log \frac{1}{T}\int_0^T d(x,\tilde{\phi}_t(x,\theta))dt}{\log T}
= \Lambda(f),
\end{equation*}
where $\Lambda(f)$ is the Lyapunov exponent of $f$ for the
Kontsevich--Zorich cocycle and with respect to the affine invariant
measure supported by the orbit closure of $S$ for the action of
$\text{SL}_2(\R)$.
\end{thm}

The proof of this proposition uses tools from dynamics on the moduli
spaces of translation surfaces. The corresponding technique was
developed by A.~Zorich in~\cite{zorich1999wind}, by G.~Forni
in~\cite{forni2013introduction} and applied in~\cite{delecroix2014diffusion}. Let
$\pi:\tilde{S}\to S$ be a lattice cover, and let $f$ be the
associated cohomology class. Let $(\tilde{x},x)$ be a pair of points
in $\tilde S$ and in $S$ respectively, such that $\pi(\tilde{x})=x$.
To study the drift of a trajectory starting from $\tilde{x}$ we
project the trajectory to the base $S$. We consider a small
horizontal interval $I$ on $S$ with one of the endpoints at $x$ and
the first return map of the vertical geodesic flow to $I$. The first
return map is generally an interval exchange transformation; the
vertical flow on $S$ can be represented by Veech's zippered
rectangles construction as a suspension flow over this interval
exchange transformation. Consider a piece of trajectory starting at
$\tilde x$ and going up to the $N$th visit to the interval $I$.
Joining the endpoints of such piece of trajectory along $I$, we get a
cycle $C_N(S,x)$ in the first homology of $S$. The drift of the
trajectory in $\tilde S$ is given by the intersection pairing
of this cycle with the cohomology class $f$:
\begin{equation}
\label{eq:f:of:Cn}
    \langle f,C_N(S,x)\rangle\,.
\end{equation}

Following Zorich and Forni,
we use a renormalization
procedure given by the Teichmüller flow on the moduli space of translation
surfaces
to study the behavior of the above quantity
for large $N$.  The Teichmüller flow $g_s$ contracts vertical
trajectories, so for an appropriate choice of time $s(N)$
cycles $C_N(S,x)$
representing long trajectories of $\phi_t$ (and thus having
large norm) are mapped to
cycles on the new surface
$g_\tau\cdot S$ having norm comparable to $1$. The growth of
expression~\eqref{eq:f:of:Cn}
is then given by the Kontsevitch--Zorich
cocycle governed by the Gauss--Manin connection on the
Hodge bundle.
If we fix a cross section for the Teichmüller flow in the moduli space
of translation surfaces, we obtain a sequence of flat surface
$S_n$ in the cross section such that
\begin{equation*}
    g_{s_n}\cdot S =\Gamma_n \cdot S_n\,,
\end{equation*}
where $\Gamma_n$ is an element of the mapping class group. The
behavior of the expression~\eqref{eq:f:of:Cn} for large $N$ can be related to the
behavior of the sequence $\Gamma_n^*f$, which is given by the
Oseledets theorem.

In Section~\ref{s:Background} we recall some classical results in
the theory of translation surfaces. Namely, we recall the definition of
moduli spaces of translation surfaces, the action of $\text{SL}_2(\mathbb{R})$
and we define the Kontsevich--Zorich cocycle on the Hodge bundle. We also
recall fundamental theorems on dynamics on moduli spaces of
flat surfaces.

In Section~\ref{s:Large:Excursions} we prove Theorem~\ref{SB}. Using
Veech's zippered rectangles construction presented in
Section~\ref{paragraph_coverings_estimate},
we relate the statement of the
Theorem to the behavior of the sum
\begin{equation*}
   \sum_{k\le \n} \langle f,C_k(S,x)\rangle
\end{equation*}
for large $\n$. We define a convenient cross-section for the
Teichmüller flow in Section~\ref{paragraph_cross_section}. In
Section~\ref{paragraph_large_excurssion} we use this cross section
and the renormalization procedure to construct trajectories with a
good excursion rates by applying the Oseldet's theorem. In
Section~\ref{paragraph_average} we establish the upper bound in
Theorem~\ref{SB} by using the Oseldet's theorem. To obtain the lower
bound we use the trajectories constructed in
Section~\ref{paragraph_large_excurssion} and a uniform upper bound
to control the diffusion near these excursions. Finally, a proof of a
technical Lemma~\ref{ie} is placed to an Appendix.

\subsection*{Acknowledgements.}
I am very grateful to Pascal Hubert for supporting me during this
work at the I2M of Marseille, and introducing me to Teichmüller
dynamics and its applications. I am also very grateful to Anton Zorich for a careful read of this paper.

\section{Background facts and constructions}
\label{s:Background}
The present paper uses tools from the theory of compact flat
surfaces. We refer to the surveys of A. Zorich \cite{zorich2006flat}
and of J.C. Yoccoz \cite{yoccoz2007interval} for an elementary
introduction to the subject and a survey~\cite{forni2013introduction}
of G. Forni and C. Matheus for more details. In this section, we
introduce the material which will be used in the proof of
Theorems \ref{SB1} and \ref{SB} announced in the introduction. We
start by presenting some basic facts on flat surfaces and lattice
coverings. We also define the Teichmüller and moduli spaces of flat
surfaces and give recent results related to the dynamics of the
action of $\text{SL}_2(\R)$. Finally we recall Veech's zippered
rectangle construction which we use in the study of the geodesic flow
of a flat surface.

\subsection*{Compact flat surfaces}
A compact translation surface $S$ is commonly defined as a pair
$(S,\omega)$ where $S$ is a compact Riemann surface and $\omega$ is a
non-zero holomorphic differential on $S$. The one-form $\omega$
defines a flat metric $ds^2=\text{Re} (\omega \otimes \bar{
\omega})$. Compactness of the surface $S$ implies that the induced
area of $S$ is necessarily finite.

The metric as above has singularities
at the finite set $\Sigma_{\omega}$ of zeros of $\omega$;
at a point $x\in \Sigma_{\omega}$ the metric has a conical
singularity of angle $2\pi (k_x+1)$, where $k_x$ is the order of the
zero. If $g$ is the genus of the surface, the Gauss--Bonnet formula
gives
   
\begin{equation*}
    \sum_x k_x =2g-2\,.
\end{equation*}

By taking local primitives of $\omega$ in simply-connected domains in
the complement to $\Sigma_{\omega}$, we get an atlas of charts on
$S\backslash \Sigma_{\omega}$ with values in $\C$ and such that the
transition functions are translations.

The geodesic flow is defined on the tangent bundle to $S\backslash
\Sigma_{\omega}$. A trajectory may land to
a singularity in a finite time in onward or in a backward direction.
Such trajectories are called \textit{singular trajectories}.

The translation structure allows
to trivialize the tangent
bundle over $S\backslash \Sigma_{\omega}$. Thus,
a geodesic on a translation surface
preserves the initial direction, which allows to
consider a \textit{directional flow}
in any fixed direction $\theta$.
We mainly consider the vertical (pointing to the North) and the
horizontal (pointing to the Est) directional flows. The geodesic flow on
a flat surface has been studied extensively in the past decades.
The following fundamental result concerning the
directional flows is due to S.~Kerckhoff, H.~Masur and
P.~Smillie~\cite{kerckhoff1986ergodicity}

\begin{theo}[Kerckoff, Masur, Smillie]
\label{KMS}
For every translation surface $S$,
the directional flow in direction $\theta$
is uniquely ergodic for almost all directions $\theta$.
\end{theo}

A saddle connection on a flat surface is a singular trajectory which
connects two singularities (possibly the same). We are particularly
interested in the vertical and horizontal saddle connections. The
following result from~\cite{keane1975interval}
is often called \textit{Keane's  criterion}.

\begin{theo}[Keane \cite{keane1975interval}]
If a flat surface has no saddle connections in direction $\theta$,
then the geodesic flow in this direction is minimal.
\end{theo}

\subsection*{Lattice coverings}
\label{lattice_cover}
A \textit{lattice cover of a translation surface} is a triple
$(\tilde{S},S,\pi)$. Where $S$ is a compact translation surface,
$\tilde{S}$ is a connected translation surface, and
$\pi: \tilde{S}\rightarrow S$ is a holomorphic $\Z^d$-covering map. Moreover, we assume 
that the holomorphic one-form $\tilde\omega$,
defining the translation structure
on $\tilde S$, is the pullback of the holomorphic one-form $\omega$
defining the
translation structure on $S$. That is $\pi^* \omega= \tilde{\omega}$.
Lattice $\Z$-covers were studied in~\cite{hooper2012generalized}
and~\cite{hooper2013dynamics}. For instance, we can think to a surface embedded in an highter dimentionnal torus and periodic in d-directions.

By definition, the monodromy group
of a lattice cover is
isomorphic to $\Z^d$.
We denote his action
on $\tilde{S}$ by $\mathbf{n}\cdot x$. If we
fix a base point  $x$ on $S$, an isomorphism of the monodromy
group with $\Z^d$ defines a morphism

\begin{equation*}
f:\pi_1(S,x)\longrightarrow \Z^d.
\end{equation*}

This morphism factorizes thought the Abelianised group and defines a
class $f$ in the first cohomology of $S$:
\begin{equation*}
f\in \Hom(H_1(S,\Z),\Z^d)=H^1(S,\Z)\otimes_\Z \Z^d.
\end{equation*}
The class $f$ characterizes the cover up to an isomorphism. We
will use the cohomology class $f$ to measure the deviation of a
trajectory.

\subsection*{Teichmüller spaces}
 The Teichmüller spaces of flat surfaces are spaces of isotopy classes of
 flat surfaces with prescribed singularities (we refer to \cite{yoccoz2007interval},\cite{forni2013introduction} and \cite{zorich2006flat} for an
 introduction to the subject). We fix a compact topological surface
 $M$ of genus $g$ and a set  $\Sigma=\{x_1,...,x_r\}$ of distinct
 points in $M$. Let  $\kappa=(\kappa_1,...,\kappa_r)$  a set of non
 negative integers such that
 \begin{equation*}
     \sum_{l=1}^r \kappa_l= 2g-2,
 \end{equation*}
The Teichmüller space $\Th(\kappa)$  associated to $(M,\Sigma,\kappa)$ is the set of isotopy classes of compact marked translation  surfaces $\phi : S \rightarrow M$ such that, for all $i$ the surface $S$ have a conical singularity of angle $2\pi(\kappa_i+1)$ at $\phi^{-1}(x_i)$ (a zeros of order $\kappa_i$) and no singularities elsewhere.\\
 The Teichmüller spaces have a nice structure of complex linear
 manifolds, the local charts are given by taking period coordinates in
 the relative cohomology
\begin{equation*}
(\phi,S) \rightarrow [\phi_*\omega] \in H^1(M,\Sigma,\C).
 \end{equation*}
 The period map is a local homeomorphism (see
 \cite{yoccoz2007interval}).\\

The mapping class group $\text{Mod}(M,\Sigma)$ is the group of isotopy classes of diffeomorphisms on $(M,\Sigma)$, it acts naturally on the Teichmüller spaces. The moduli space $\h(\kappa)$ is the quotient space of
 $\Th(\kappa)$ by $\text{Mod}(M,\Sigma)$. Flat surfaces can have nontrivial automorphism groups; $\h(\kappa)$ is a complex orbifold in
general, but we can find a finite cover which is a manifold. The area of a surface define a function on the moduli and Teichmüller space
 and we denote $\Thc(\kappa),\hc(\kappa)$ the level set of surfaces
 with an area equal to one.\\
\begin{rem}[Pointed surfaces]
 We can add marked points that are regular to surfaces in  Teichmüller spaces; regular means that they do not correspond to conical singularities of the underlying surface. We call such surfaces: \textit{pointed surfaces}. For a given $\kappa$ we denote by $\Tho(\kappa)=\Th((0,\kappa))$ the Teichmüller space of flat surfaces in $\Th(\kappa)$ with a marked regular point (it's the universal curve over $\Th(\kappa)$). There is a natural projection $\Tho(\kappa)\to \Th(\kappa)$ that is invariant under the action of the mapping class group. An automorphism of a flat surface is locally given by a translation, if an automorphism stabilise a regular point then it's locally constant near this point and then
 it's globally trivial when the surface is connected. Using this we
 see that the moduli spaces $\ho(\kappa)$ of pointed surfaces are
 manifolds. 
 \end{rem}
\subsection*{Dynamics in the moduli space of translation surfaces}
In this paper we use the left action of $\text{SL}_2(\R)$ on the moduli spaces of translation surfaces $\h(\kappa)$ (see~\cite{zorich2006flat} or \cite{forni2013introduction} for a survey).
The group
acts by linear deformations of the flat structure and preserves the
area of flat surfaces then it also preserves the hypersuface $\hc(\kappa)$. The action
of the group $\text{SL}_2(\R)$ in period coordinates can be seen as
the diagonal action on the coefficients $\C\simeq\R^2$. Two
one-parameter subgroups of $\text{SL}_2(\R)$ are important in the
study of translation surfaces: the diagonal subgroup and the subgroup
$\text{SO}_2(\R)$. The flow induced by the diagonal subgroup is
denoted by
\begin{equation}
g_s = \begin{pmatrix} e^s & 0\\ 0 & e^{-s} \end{pmatrix}
\end{equation}
and called the \textit{Teichmüller flow}.
This flow shrinks the flat metric in the vertical
direction and dilates the metric in the horizontal direction. The flow
$r_\theta$ is defined by the action of $\text{SO}_2(\R)$:
\begin{equation*}
    r_\theta=\left(\begin{array}{rr} \cos(\theta) & -\sin(\theta)\\
    \sin(\theta) & \cos(\theta) \end{array}\right).
\end{equation*}
It does not change the flat metric, but it changes
the distinguished direction (``direction of the North'').
  

The vector space $H^1(S,\Sigma;\C)$ serving for period
coordinates admits a natural integer lattice
$H^1(S,\Sigma;\Z\oplus i\Z)$, which endows it with a
canonical volume element: the one, for which a fundamental domain of
the lattice has a unit volume. This volume element defined in period
coordinates induces a well-defined volume element in every stratum
$\mathcal{H}(\kappa)$ of translation surfaces called the
\textit{Masur--Veech} volume element. By construction the group
$\text{SL}_2(\R)$ preserves the Masur--Veech volume element and the flat area of the surfaces. Thus, the Masur--Veech volume element in $\h(\kappa)$
induces a natural volume element on the hypersurface
$\hc(\kappa)$. The following Theorem
proved independently by H.~Masur~\cite{masur1982interval} and by
W.~Veech~\cite{veech1978interval} is the keystone result in dynamics in the
moduli space of translation surfaces.

\begin{theo}[Masur, Veech]
The total volume of any stratum $\mathcal{H}^1(\kappa)$ in the moduli space
of compact translation surfaces is finite. The Teichmüller flow acts
ergodically on every connected component of every stratum.
\end{theo}

\subsection*{Orbit closures and affine invariant submanifolds}
Let $S$ be a translation surface of area $1$ in some stratum
$\mathcal{H}(\kappa)$,  $S\in \hc(\kappa)$. Detecting
the orbit closure
$\hc_S=\overline{\text{SL}_2(\R)\cdot S}$ (resp $\h_S=\overline{\text{GL}_2(\R)\cdot S}$) of the orbit of $S$ in
$\hc(\kappa)$ is one of the most important questions in the theory of
translation surfaces.
The Magic Wand theorem of
A.~Eskin--M.~Mirzakhani--A.~Mohammadi~\cite{eskin2013invariant},
\cite{eskin2015isolation}
(and an important complement by S.~Filip~\cite{filip2016splitting})
provide the general structure of
$\hc_S$.

\begin{theo}[Eskin--Mirzakhani--Mohammadi]
For each translation surface $S$ the orbit closure
$\h_S$
is locally a linear subspace in period coordinates
(possibly a finite collection of linear subspaces).

The orbit closure $\h_S$
admits a measure $\tilde{\mu}_{S}$ such that
 \begin{itemize}
\item Locally, the measure $\tilde{\mu}_S$ is proportional to the Lebesgue measure in the period coordinates.
\item The measure $\mu_S$ induced on $\hc_S$ is an invariant probability measure for the Teichmüller flow.
\item The measure $\mu_S$ is ergodic for the Teichmüller flow.
\end{itemize}
\end{theo}


\subsection*{Hodge bundle and Kontsevich--Zorich cocycle}
 %

Every stratum $\h(\kappa)$ admits two natural vector bundles $\hat{H}^1\h(\kappa)$ and  $H^1\h(\kappa)$ with
fibers $H^1(S,\Sigma_S,\R)$ and $H^1(S,\R)$ over a point in
$\h(\kappa)$ represented by a translation surface $S$. Both bundles
carry a natural flat connection called the \textit{Gauss--Manin
connection} uniquely determined by the property of equivariance of
lattices $H^1(S,\Sigma_S,\Z)$ and $H^1(S,\Z)$ respectively. The
\textit{Kontevich--Zorich cocycles} correspond to the monodrony of
these connections  along the $\text{SL}_2(\R)$-action. These
cocycles, which we denote by $G_s$, serve for an efficient
renormalization procedure.

The bundle $H^1\h(\kappa)$ is endowed with a natural norm $\|.\|_{h}$
called the \textit{Hodge norm}; this norm is induced from the
classical Hodge norm on $H^{1,0}(S,\C)$ by the canonical isomormhism
\begin{equation*}
H^1(S,\R) \longrightarrow H^{1,0}(S,\C)\,.
\end{equation*}
G.~Forni proved in~\cite{forni2013introduction} the following uniform
upper bound for the growth rate of the Hodge norm along the
Teichmüller flow:
\begin{theo}[Forni, \cite{forni2013introduction}]
For every $u\in H^1\h(\kappa)$ such that $\|u\|_{h}=1$ the following
bound holds:
\begin{equation*}
\frac{d\| G_s u \|_{h}}{ds}\Big\vert_{s=0} \le 1\,.
\end{equation*}
\end{theo}
This bound implies, in particular,
integrability of the Kontsevich--Zorich  cocycle for any affine
invariant probability measure on $\hc(\kappa)$.

\subsection*{Oseledets and Birkhoff ergodic theorems}
A.~Eskin and J.~Chaika have proved in~\cite{chaika2015every}
the following
important versions of the Birkhoff and
Oseldets ergodic theorems for the measure $\mu_S$ and the
Kontsevich--Zorich cocycle. These theorems are particularly efficient
in the study of translations surfaces (see
also the survey of J.~Chaika~\cite{chaika2015every}).
\begin{theo}[Chaika--Eskin]
\label{CE1}
For any translation surface $S$
in any stratum $\hc(\kappa)$,
for almost any angle $\theta$ and for any
continuous function $\phi$ with compact support in $\hc_S$
the following equality holds:
\begin{eqnarray}
\lim_{T\rightarrow +\infty} \frac{1}{T}\int_0^{T} \phi(g_s r_\theta S)ds
= \int_{\hc(S)} \phi ~d\mu_S\,.
\end{eqnarray}
\end{theo}

\begin{theo}[Chaika--Eskin]
\label{CE2}
For any translation surface $S$ in any stratum $\hc(\kappa)$, and for
almost any angle $\theta$, the translation surface $r_\theta S$ is
Oseledets-generic for the Kontsevich--Zorich cocycle along the
Teichmüller flow with respect to the measure $\mu_S$.
\end{theo}

\begin{rem}[cocycles with multiple components]
\label{rem_cocycle}
Recall that a cohomology class $f\in H^1(S,\Z)\otimes_\Z \Z^d$
associated to a $\Z^d$-lattice cover has $d$ components
$f=(f_i)_{1\le i\le d}$ and takes values in $\Z^d$. In what follows
the Hodge-norm of such a cohomology class is defined as the sup-Hodge
norm of components $\|f\|_h:=\max_i \|f_i\|_h$. Thus, the Lyaponov
exponent of $f$ with respect to the Kontsevich--Zorich cocycle 
the along the Teichmüller flow is the maximum of the Lyapunov
exponents of the components $\Lambda(f) = \max_i \Lambda(f_i)$.
\end{rem}
\subsection*{Zippered rectangles construction}
\label{paragraph_zippered}
Veech's zippered rectangles construction is a particularly convenient
way to represent a vertical flow on a compact translation surface,
see \cite{veech1978interval} for the
original definition or~\cite{yoccoz2007interval} for a survey. We
recall briefly an idea of the construction. Consider a horizontal
segment $I$ on a translation surface $S$. We always assume that $I$
does not have conical singularities inside it (a conical singularity
at one of the extremities of $I$ is allowed and sometimes even
assumed). The first return map $T: I\to I$ of the vertical flow on
$S$ to $I$ is an \textit{interval exchange transformation}: it chops
the interval $I$ into several subintervals and rearranges them in a
different order, preserving the orientation and avoiding overlaps (see
the surveys~\cite{yoccoz2007interval} and~\cite{forni2013introduction},
\cite{viana2008dynamics} for details on interval exchange
transformations). More formally, we have two decompositions of $I$: $I^t=\sqcup_\alpha I_\alpha^t,~~~~I^b=\sqcup_\alpha I_\alpha^b$ indexed by a set $\mathcal{A}$; here $I\backslash I^{\cdot}$ is a finite set of singularities. The map $T$ sends $I_\alpha^t$ to $I_\alpha^b$ by a translation; we generally denote $\lambda_\alpha$ the length of $I_\alpha^{\cdot}$. We denote $I^*$ the set of points with an infinite future orbit; the map $T: I^*\to I^*$ is well defined. 
  %
  %
  %
The return time of the vertical flow is piecewise constant on $I$ and
constant on each of subintervals $I_\alpha^t$, we denote it $\tau_\alpha$. Thus, a choice of the horizontal
subinterval $I$ as above defines a decomposition of the translation
surface into a collection of rectangles --- flow boxes $I_\alpha^t\times [0,\tau_\alpha[$ of the
vertical flow. This decomposition changes, of course, when we change the
base interval $I$. We will mostly assume that $I$ has unit length
and that the translation surface $S$ has unit area.
  %

\subsection*{Coding the vertical geodesic flow}
Let $S$ be a translation surface, $I$ be a horizontal segment
of length one
on $S$ as above and $T:I\to I$ the first return map of the vertical flow
to $I$. We assume that the left endpoint $x_0$ of $I$ is nonsingular.
By $S'$ we denote a pair $S'=(S,x_0)$.
For any positive integer $n$ and $x\in I$
we construct $n$-th return cycle
$C_\n(S',x)$
as follows. Consider an oriented segment $\gamma_\n(x)$
of the trajectory
of the vertical flow starting from
$x$ and going up to $T^\n(x)$. Completing the resulting path
to a closed path by a subsegment of $I$ starting from $T^\n(x)$
and ending at $x$ we get a cycle
\begin{equation*}
     C_\n(S',x)\in H_1(S,\Z),
\end{equation*}
with $S'=(S,x_0)$ (more precisely, the cycles can be lifted to the homology of $S\backslash \Sigma_{S'}$ with $\Sigma_S$ the set of singularities).  This is an additive cocycle for the interval exchange map $T$, it satisfies
\begin{equation*}
     C_{\n+1}(S',x)=C_\n(S',x)+C_1(S',T^\n x).
\end{equation*}
We also denote $C_\n(S')$ the return cycle of the trajectory starting
from $x_0$, it satisfies
\begin{equation*}
     C_\n(S')=\underset{x\rightarrow x_0}{\lim}C_\n(S',x).
\end{equation*}
Note that for all interior points $x\in I_\alpha^{t}$ the first return cycles $C_1(S',x)$ are homologous. We
denote by $h_\alpha$ the resulting first return cycles. When the
vertical flow is minimal, these cycles form a basis of  $
H_1(S\backslash \Sigma_{S'},\Z)$.


\section{Large excursions}
\label{s:Large:Excursions}

In this section, we prove the main theorems of the paper. We first show briefly how we can reduce the problem from lattice coverings of translation surfaces to compact translation surfaces. This was previously discussed in the paper \cite{delecroix2014diffusion}. In this same section, we rephrase the theorem \ref{SB} by using the formalism of zippered rectangle construction \ref{paragraph_zippered}. In a second, by using Veech zippered rectangles construction, we choose a convenient cross section for the Teichmüller flow and then consider the first return map on this section. To define the cross section, we use the lemma \ref{ie} which is given in the appendix. In the third section, by using the normalization procedure, we produce a sequence of trajectories with a good diffusion rate, and we show that these excursions occur at exponential times. In fourth, from the results of A. Zorich \cite{zorich1997deviation}, we obtain an upper bound for the diffusion. We use this upper bound to control the trajectories around a large excurssion and also to state the upper bound in theorem \ref{SB}. Finally, we prove the lower bound, which is the most tricky part. By using the excursions, we can show that the trajectory goes far from the initial point; with the upper bound, we can control the diffusion near these excursions and show that a typical trajectory spends a large part of its time far from the starting point. 

\subsection*{Preliminary results on coverings}
\label{paragraph_coverings_estimate}
Let $\pi : \tilde{S}\rightarrow S$ be a lattice covering of a translation surface $S$ with cocycle $f$. The flat structure defines a metric on $\tilde{S}$, and we denote $d_{\tilde{S}}$ the flat distance on $\tilde{S}$. Let $x,y$ be two points in the same fiber of $\pi$. Using the $\Z^d$ action, there is a unique  $\mathbf{n}\in \Z^d$ such that $y=\mathbf{n}\cdot x$. The lemma below allows us to compare the geodesic distance between $x,y$ and the norm of $\mathbf{n}$ (for instance, $\|\mathbf{n}\|=\sum_i |\mathbf{n}_i|$).

\begin{lem}
\label{lem_distance_fiber}
For all lattice covering $\tilde{S}\to S$, with a choice of cocycle $f$. There exists a constant $A_0>0$ such that: for every $\tilde{x}\in \tilde{S}$ and $\mathbf{n}\in \Z^d$, we have 
\begin{equation*}
A_0^{-1}\|\mathbf{n}\| \le d_{\tilde{S}}(x,\mathbf{n}\cdot x) \le A_0\|\mathbf{n}\|.
\end{equation*}
\end{lem}

Using this lemma, we can approach the deviation of the geodesic flow on $\tilde{S}$. To do this, we use Veech zippered rectangle construction for the surface $S$. Let $\tilde{x}\in \tilde{S}$, we denote $\tilde{\phi}_t(\tilde{x})$ the geodesic flow on $\tilde{S}$ in the vertical direction. Let $x=\pi(\tilde{x})$ be the projection on $S$. We assume that $x$ has an infinite future orbit, and $S$ is constructible by zippered rectangles with the interval $I$ of length $1$ starting from $x$ and going to the east. As before, we denote $S'=(S,x)$ and $\phi_t(x)$ the vertical flow on $S$. Let $(t_{\n})_{\n\ge 0}$ be the sequence of return times of $(\phi_t(x))_{t\ge 0}$, on the interval $I$, and $C_{\n}(S')$ be the $\n-$th return cycle of the trajectory starting from $x$. 

\begin{lem}
\label{lem_covering_gathering}
Under these assumptions, there exist $A_1,B_1$ such that, for all $\n>0$ 
\begin{equation*}
\frac{A_1^{-1}}{\n} \underset{k\le \n}{\sum}\|\langle f , C_{k} \rangle\| -B_1\le \frac{1}{T}\int_0^T d(\tilde{x},\tilde{\phi}_{t}(\tilde{x}))dt \le  \frac{A_1}{\n} \underset{k\le \n}{\sum} \|\langle f , C_{k} \rangle\| +B_1, ~~~\forall T\in [t_\n,t_{\n+1}].
\end{equation*}
\end{lem}

\begin{proof}
Let $\mathbf{n}\in \Z^d$ and let $\gamma$ be any arc connecting $x$ and $\mathbf{n}\cdot x$. The vector $\langle f,\pi_*\gamma\rangle$ is then equal to $\mathbf{n}$. For any $\n\in \N$, the point $\tilde{\phi}_{t_\n}(\tilde{x})$ belongs to a unique leaf of $I$. If $\tilde{I}$ is the leaf starting from $\tilde{x}$, then $\tilde{\phi}_{t_\n}(\tilde{x})$ belongs to $\mathbf{n}\cdot \tilde{I}$, where $\mathbf{n}= \langle f, C_{\n} \rangle $. By using the triangular inequality, we can obtain the following inequality: 
\begin{equation*}
A_0^{-1}\|\langle f , C_{\n} \rangle\| -1 \le d(x,\tilde{\phi}_{t_\n}(\tilde{x})) \le A_0\|\langle f , C_{\n} \rangle\| +1,
\end{equation*}
where $1$ is the length of the segment $I$. Using triangular inequality again, we can write
\begin{equation*}
|d(x,\tilde{\phi}_{t_\n}(\tilde{x}))-d(x,\tilde{\phi}_{t}(\tilde{x}))|\le d(\tilde{\phi}_{t_\n}(\tilde{x}),\tilde{\phi}_{t}(\tilde{x})) \le |t-t_\n|.
\end{equation*}
The last inequality uses the fact that the velocity is constant and equal to one. The increments $(t_{k+1}-t_{k})_{k\le \n}$ take their values in a finite set of strictly positive numbers (because in the Veech zippered rectangle construction, the return time is constant on each sub-interval of $I$). Then there is $B_0>0$ with 
\begin{equation*}
   B_0^{-1} \le t_{k+1}-t_{k}\le B_0,~~~~\forall k\ge 0,
\end{equation*}
and this implies 
\begin{equation*}
    \n B_0^{-1}\le T\le B_0(\n+1),~~~~\forall T \in [t_\n,t_{\n+1}[.
\end{equation*}
Using these inequalities, we can estimate the integral by splitting the interval $[0,T]$ into subintervals $[t_k,t_{k+1}[$ for $k< \n$ and $[t_\n,T]$. 
We have 
\begin{eqnarray*}
 \frac{1}{T}\int_0^Td(\tilde{x},\tilde{\phi}_{t}(\tilde{x}))dt&\le& \frac{1}{T}\sum_{k< \n}\int_{t_k}^{t_{k+1}}d(\tilde{x},\tilde{\phi}_{t}(\tilde{x}))dt~~+~~ \frac{1}{T}\int_{t_\n}^{T}d(\tilde{x},\tilde{\phi}_{t}(\tilde{x}))dt\\
 &\le&\frac{1}{T}\sum_{k< \n}\int_{t_k}^{t_{k+1}}d(\tilde{x},\tilde{\phi}_{t_k}(\tilde{x}))+(t-t_k)dt~~+~~\frac{1}{T}\int_{t_\n}^{T}d(\tilde{x},\tilde{\phi}_{t}(\tilde{x}))dt\\
 &\le& \frac{1}{T}\sum_{k< \n} (t_{k+1}-t_k)(A_0 \|\langle f , C_{k} \rangle\| + (t_{k+1}-t_k))~+~t_{\n+1}-t_{\n}\\
    &\le&\frac{A_0B_0}{T} \underset{k< \n}{\sum}\|\langle f , C_{k} \rangle\|+\frac{\n B_0^2}{T}+B_0\\
    &\le& \frac{A_0B_0^2}{\n} \underset{k< \n}{\sum}\|\langle f , C_{k} \rangle\|+B_0^3+B_0
\end{eqnarray*}
And this implies the existence of two constants $A_1,B_1$ such that, for all $\n$ we have
\begin{equation*}
 \frac{1}{T}\int_0^Td(x,\tilde{\phi}_{t}(\tilde{x},\theta))dt \le \frac{A_1}{\n} \underset{k\le \n}{\sum} \|\langle f , C_{k}(S') \rangle\| +B_1, ~~~\forall T\in [t_\n,t_{\n+1}[.
\end{equation*}
Similarly, we can obtain the lower bound and prove the lemma \ref{lem_covering_gathering}.
\end{proof}

If we are under the assumption of the last lemma and if $\Lambda f>0$. To prove the theorem \ref{SB}, it's sufficient to prove the following: for all $\epsilon>0$, it exists $\n_0$ such that
\begin{equation}
\label{formule_intermediaire}
\exp((\Lambda f+1-\epsilon)\log(n)) \le \underset{k\le \n}{\sum} \|\langle f , C_{t_k} \rangle\| \le \exp((\Lambda f+1+\epsilon)\log(\n)),~~~~\forall \n\ge \n_0.
\end{equation}
We will prove this statement in the rest of the paper. The following remark in important.
\begin{rem}[Restriction to the case $d=1$]
    By using remark \ref{rem_cocycle} restrict to cocycle $f$ with values in $\Z$, if $f=(f_i)_{i=1...d}$, the Lyapunov exponent is $\max_i \Lambda f_i$ and also for the diffusion rate.  
\end{rem}

\subsection*{Cross section for the Teichmüller flow}
\label{paragraph_cross_section}
We fix a pointed surface $S'=(S,x_0)$ of area one and let $\mu_{S'},\mu_S$ be the two affine measures on the orbits $\hc(S),\hc(S')$. Up to rotation and the action of the Teichmüller flow, we can obtain a new surface $S'_1$ with no vertical saddle connection and built by zippered rectangles using a segment of length one. Let $(\mathcal{A},\pi,\lambda_0,\tau_0)$ be the data for the zippered rectangles construction. We denote $\lambda_0=(\lambda_0(\alpha))_\alpha$ the vector whose entries are the length of the intervals, and $\tau_0=(\tau_0(\alpha))_\alpha$ the height of the rectangles in the zippered rectangles construction. These data allow us to construct a pointed translation surface that is canonically isomorphic to $S_1'$. The construction also gives canonical basis  $(h_\alpha)_{\alpha \in \mathcal{A}}$ and $(\zeta_\alpha)_{\alpha \in \mathcal{A}}$ of the homology groups:
\begin{equation*}
    H_1(S \backslash \Sigma_{S'},\mathbb{Z}),~~~\text{and}~~~H_1(S,\Sigma_{S'},\mathbb{Z}),
\end{equation*}
see \cite{yoccoz2007interval} for more details. For all $x\in I_\alpha^t$, we have the first return cocycle
\begin{equation*}
    C_1(S',x)=h_\alpha.
\end{equation*}
The zippered rectangles construction provides an open map from a neighborhood of $(\lambda_0,\tau_0)$ in $\R_{>0}^{\mathcal{A}}\times \R_{>0}^{\mathcal{A}}$ to $\ho(\kappa)$. We can choose a neighborhood $U$ of $(\lambda_0,\tau_0)$ that satisfies the following properties:
\begin{itemize}
\item There exists $\delta >0$ such that $U=\{g_s(\lambda,\tau) ~,~(s,(\lambda,\tau))\in (-\delta,\delta)\times V\}$ where $V$ is the cross section formed by surfaces of $U$ such that $\|\lambda\|_1=1$ (with $g_s(\lambda,\tau)=(e^s\lambda,e^{-s}\tau)$).
\item There exists $C_1>0$ such that, for all $(\lambda,\tau) \in U$:
\begin{equation*}
 C^{-1}_1 \le \lambda_\alpha \le C_1,~~~\text{and}~~~C^{-1}_1 \le \tau_\alpha \le C_1,~~~\forall \alpha \in \mathcal{A}.
\end{equation*}
\end{itemize}
The space $\ho(\kappa)$ of pointed translation surfaces is a manifold, then we can reduce $V,\delta$ and assume that the map $U\rightarrow \ho(\kappa)$ is an embedding, and then the return time on $V$ of the Teichmüller flow is bounded from below by $2\delta$.\\
We still denote $U,V$ the image of $U,V$ in $\ho(\kappa)$ and $U^{(1)},V^{(1)} \subset \hoc(\kappa)$. For all surfaces $Y'=(Y,y_0)\in U$, there is a canonical representation of $Y'$ using the same model as $S'_1$, and then we have trivialization of the homology group of all surfaces in $U$. We use the following notations:
\begin{itemize}
\item $I(Y')$ canonical segment of $Y'$ used for the zippered rectangles construction,
\item $w_Y$ is the Abelian differential defining $Y$,
\item $C_\n(Y',y)$ the $\n-th$ return cycle on $I(Y')$ for $y\in I(Y')$,
\item $C_\n(Y')$ the $\n-th$ return cycle of the vertical leaf starting from the marked point $y_0$,
\item $\|c\|_1= \sum_\alpha |\langle \zeta_\alpha , c\rangle|,~~\forall c \in H_1(Y \backslash \Sigma_{Y'},\mathbb{R})$.
\end{itemize}
\begin{rem}
For all $\n,\alpha,y$, we have $\langle \zeta_\alpha, C_\n(Y',y) \rangle\ge 0$, when it's defined.
\end{rem}
We also need the following technical lemma, let $I^*(Y)$ be the subset of points in $I(Y)$ with an infinite future orbit.

\begin{lem} 
\label{ie}
If the vertical flow of $S'_1$ is minimal, then there exists $l$ such that all the trajectories of length larger than $l$ cross all the cycles $\zeta_\alpha$ at least once, in other words.
\begin{equation}
\langle \zeta_\alpha , C_\n(S'_1,x) \rangle > 0,~~~~~\forall \n\ge l,~~~\forall x\in I^*(S'_1),~~~\text{and}~~~\forall \alpha \in \mathcal{A}.
\end{equation}
For all $l\ge 0$, by reducing $V$, we can assume for all $Y'\in V$, and for all $\n \le l$:
\begin{equation}
C_\n(Y')=C_\n(S'_1).
\end{equation}
\end{lem}

The proof of this lemma uses the formalism of intervals exchange, and it's given in the appendix. By assumption, we can choose an integer $l$ that satisfies the first part of the lemma and reduce $V$ so that the second statement is also satisfied (we keep the same notation $V$ for this smaller cross section). As we have a uniform lower bound for the return time in $V^{(1)}$, the measure on $V^{(1)}$ induced by $\mu_{S'}$ is finite. We denote by $\sigma$ the volume of $V^{(1)}$, and $\nu$ the measure induced on $V^{(1)}$.

The following result is a corollary of the Birkhoff theorem  (theorem \ref{CE1}).
\begin{cor}
\label{CE11}
For almost every $\theta$, the following limit is true:
\begin{equation*}
    \lim_{s \rightarrow +\infty} \frac{\#\{u\le s,~~~~ g_u r_\theta S' \in V \}}{s}= \sigma.
\end{equation*}
In the LHS we count the number of visits in the cross section on the interval $[0,s]$.
\end{cor}

\subsection*{Large excursion}
\label{paragraph_large_excurssion}
For almost every $\theta$, the surface $r_\theta S'$ is generic for the corollary \ref{CE11}, and then the trajectory $(g_s r_\theta S' )$ visits $V$ an infinite number of times. We denote by $(S^{(n)}_\theta)_{n\ge 0}$ the sequence of visits and $S_\theta=S^{(0)}_\theta$ the first one, and we still denote $f$ the cocycle on $S_\theta$ (we drop the $'$ for simplicity). We prove the following result, which allows us to construct trajectories with a good diffusion rate.
\begin{prop}
\label{LDC}
If $\Lambda(f)>0$, for almost every $\theta$, for all $\epsilon > 0$, there exists a sequence $\n_n$ 
\begin{equation}
| \langle f, C_{\n_n}(S_\theta) \rangle | \ge \exp ((\Lambda f-\epsilon)\log \n_n).
\end{equation}
Moreover
\begin{equation}
\lim_n \frac{\log \n_n}{n} = \beta := \frac{1}{\sigma}.
\end{equation}
\end{prop}

Using lemma \ref{lem_distance_fiber} on lattice covering, we can deduce the following corollary:

\begin{cor}
If $\pi : \tilde{S'} \rightarrow S'$ is a lattice covering of a compact translation surface defined by a cocycle $f$, with $\Lambda f>0$. For almost every $\theta$ and for all $\epsilon>0$, there exists a sequence of times $(t_n)_{n\ge 1}$ such that
\begin{equation*}
\lim_n \frac{\log t_n}{n} = \frac{1}{\sigma},
\end{equation*}
and
\begin{equation*}
\liminf_n \frac{\log d(\tilde{x},\tilde{\phi}_{t_n}(\tilde{x},\theta))}{\log t_n} = \Lambda f - \epsilon.
\end{equation*}
Where $\tilde{x}$ is the marked point.
\end{cor}

We give the proof of proposition \ref{LDC}.
\begin{proof}
We fix $\theta$ 
In the marked Teichmüller space, we have $S^{(n)}_\theta = g_{s_n}\Gamma^{(n)}(S_\theta) S_\theta$, where $\Gamma^{(n)}(S_\theta)$ is a homeomorphism and $s_n:=s_n(S_\theta)$ is the n-th return time on $V^{(1)}$ of $(g_s(S_\theta))_{s\ge 0}$. If $I^{(n)}(S_\theta)=e^{-s_n} I(S_\theta)$, and let $N^{(n)}_\alpha:=N^{(n)}_\alpha(S_\theta)$ be the return time on $I^{(n)}(S_\theta)$ of the trajectory $(T^k_{S_\theta} y)_{k\ge 0}$ of a point $y\in I^{(n)}_\alpha(S_\theta)=e^{-s_n}I_\alpha(S^{(n)}_\theta)$, and $r_k^{(n)}:=r_k^{(n)}(S_\theta)$ the k-th return of the trajectory starting from the marked point. We have the relation
\begin{equation*}
    \langle \zeta_\alpha , C_\n(S_\theta,y)\rangle = |\{ k < \n ,~T_{S_\theta}^k(y)\in I_\alpha^t(S_\theta)\}|.
\end{equation*}
And then we have
\begin{equation*}
    \|C_\n(S_\theta,y)\|_1=\n.
\end{equation*}
The pullback by $\Gamma^{(n)}(S_\theta)$ of the first return cycle on $I(S^{(n)}_\theta)$ starting to a point in $y\in I_\alpha(S^{(n)}_\theta)$ is equal to the first return cycle on $I^{n}(S_\theta)$ starting from the point $e^{-s_n}y\in I^{(n)}_\alpha(S_\theta)$. In other words, we have the following equality:
\begin{equation*}
    C_{N^{(n)}_\alpha}(S_\theta,e^{-s_n}y)= \Gamma^{(n)}(S_\theta)^*C_1(S^{(n)}_\theta,y)=\Gamma^{(n)}(S_\theta)^*h_\alpha,
\end{equation*}
and also
\begin{equation*}
     C_{r_k^{(n)}}(S_\theta)= \Gamma^{(n)}(S_\theta)^*C_k(S^{(n)}_\theta).
\end{equation*}
We obtain the following:
\begin{eqnarray*}
N^{(n)}_\alpha&=&\| \Gamma^{(n)}(S_\theta)^*h_\alpha \|_1,~~~  \forall n,\alpha;\\
r^{(n)}_k&=&\| \Gamma^{(n)}(S_\theta)^*C_k(S_\theta^{(n)}) \|_1,~~~  \forall n,k.
\end{eqnarray*}

For all $n$, using lemma \ref{ie} if $k\le l$, we have 
\begin{equation*}
|\langle  f, C_{r^{(n)}_k}(S_\theta) \rangle | = |\langle  f, \Gamma^{(n)}(S_\theta)^*C_{k}(S_\theta^{(n)}) \rangle | = |\langle \Gamma^{(n)}(S_\theta)_* f ,  C_{k}(S_\theta) \rangle |.
\end{equation*}
Using the same lemma \ref{ie} for each $\alpha$ there exists $k<l$ such as $C_{k+1}(S_\theta)=h_\alpha + C_k(S_\theta)$. According to this, the cycles $(C_k(S_\theta))_{k\le l}$ generate the space $H_1(S\backslash \Sigma_{S'},\mathbb{R})$, then there exists a constant $C_2>0$ (that depends only of the cross section in the moduli space) such that for all cycle $c\in H_1(S ,\Sigma_{S'},\mathbb{R})$ we have
\begin{equation*}
\max_{k\le l} |\langle c , C_k(S_\theta) \rangle | \ge C_2 \|c\|_1.
\end{equation*}
And then, for all $n$, there exists $l_n:=l_n(S_\theta)$ such as
\begin{equation*}
|\langle  f, C_{r^{(n)}_{l_n}}(S_\theta) \rangle | \ge C_2 \| \Gamma^{(n)}.(S_\theta)_* f  \|_1.
\end{equation*}
As the cross section is relatively compact in the Teichmüller space, we can compare the Hodge norm and the $\|.\|_1$ norm on the Hodge bundle restricted to the cross section (by pullback, $f$ induces an element of $H_1(S\backslash \Sigma_{S'},\R)$, and then we can define $\|f\|_1=\sum_\alpha |\langle f,h_\alpha\rangle|$ on $H^1(S,\R)$). There exists a constant $C_3 >0$ such that for all $f\in H_1(S,\R)$:
\begin{equation*}
\| f\|_1 \ge C_3 \|f\|_h.
\end{equation*}
 Let $\n_n=r^{(n)}_{l_n}$, for almost every $\theta$, by the theorem \ref{CE2} we have 
\begin{equation*}
\frac{\log  \| \Gamma^{(n)}(S_\theta)_* f  \|_1}{ n} \ge \frac{\log  \| G_{s_{n}}  (r_\theta \cdot S', f)\|_h}{ n}+ \frac{\log C_3}{n}.
\end{equation*}
Then, by applying the Oseldet theorem to $(S',f)$, we have for almost every $\theta$:
\begin{equation*}
\lim_n \frac{\log  \| G_{s_n}  r_\theta \cdot S', f)\|_h}{s_{n}} = \Lambda f.
\end{equation*}
And then, as
\begin{equation*}
\lim_{n \rightarrow + \infty} \frac{s_n}{n}= \beta,
\end{equation*}
we have
\begin{equation*}
\lim_{n \rightarrow + \infty} \frac{\log  \| \Gamma^{(n)}(S_\theta)_* f  \|_1}{n}=\Lambda f \beta.
\end{equation*}
To conclude, we use a lemma extracted from \cite{zorich1999wind}:
\begin{lem}
\label{lemzo1}
For almost every $\theta$ and for all $\alpha$
\begin{equation*}
\lim_{n \rightarrow +\infty} \frac{\log N^{(n)}_\alpha(S_\theta)}{n}=\beta.
\end{equation*}
\end{lem}
We can see that there is a sequence of strictly positive integer $(a(\alpha))_\alpha$ independent of $n$, with 
\begin{equation*}
    \n_n=\sum_\alpha a(\alpha) N_\alpha^{(n)}(S_\theta)~~~~~~\sum_\alpha a(\alpha)=l,
\end{equation*}
and we can use this to obtain 
\begin{equation*}
  l \max_\alpha N_\alpha^{(n)} \le \n_n\le l \min_\alpha N_\alpha^{(n)}.
\end{equation*}
And finally, the lemma gives
\begin{equation*}
    \lim_{n \rightarrow +\infty} \frac{\log \n_n}{n}=\beta,
\end{equation*}
and then we have 
\begin{equation*}
    \lim_{n \rightarrow +\infty} \frac{\log  \| \Gamma^{(n)}(S_\theta)_* f  \|_1}{\log \n_n}=\Lambda f.
\end{equation*}
\end{proof}
\subsection*{Average of the diffusion}
\label{paragraph_average}
Here we prove theorem \ref{SB}. We use the results of A. Zorich \cite{zorich1999wind} in order to establish a uniform upper bound. In a second time, we use the proposition \ref{LDC} and the uniform upper bound to give a lower bound, and this is enough to prove the proposition \ref{SB}.

\textbf{Uniform upper bound:}
For the last section, we need a uniform upper bound, which can be obtained using the same techniques as in the Zorich paper's \cite{zorich1999wind}. We do not give proof of this result here.
\begin{prop}
\label{upbound1}
For almost all $\theta$, for all cocycle $f$ such that $\Lambda f>0$, and for all $\epsilon>0$, there exists $N_0$ such that
\begin{equation*}
    \frac{\log|\langle f , C_\n(S_\theta,x)\rangle|}{\log n} \le \Lambda(f)+\epsilon,~~~\forall \n \ge N_0.
\end{equation*}
The bound is true for all $x\in I^*(S_\theta)$ (i.e. points with an infinite future orbit).
\end{prop}
The important point in this proposition is the fact that the bound is uniform on the segment $I^*(S_\theta)$, we will use it in the next section. We can obtain the following corollary by performing a summation of the last inequality.
\begin{cor}
\label{upbound2}
Let $\theta$ and $f$ that satisfy the statement of the last proposition. Then, for all $\epsilon>0$, there exists $N_0$ such that
\begin{equation*}
\sum_{k\le \n} |\langle f, C_k(S_\theta,x)\rangle| \le  \exp\left ( \left(\Lambda(f)+1+\epsilon\right) \log \n \right),~~~\forall \n \ge N_0,
\end{equation*}
and the bound is true for all $x\in I^*(S_\theta)$.
\end{cor}
\begin{proof}
    Let $\theta$ satisfying the assumption of the last proposition. For all $\epsilon>0$, according to the last proposition, there is $N_0\ge 0$ such as: for all $\n\ge N_0$ and all $x\in I^*(S_\theta)$ we have 
    \begin{equation*}
        |\langle f, C_\n(S_\theta,x)\rangle|\le \n^{\Lambda f +\frac{\epsilon}{2}}
    \end{equation*}
Moreover, by using the fact that $C_\n$ is an additive cocycle we have the trivial bound:
    \begin{equation*}
        |\langle f, C_\n(S_\theta,x)\rangle|\le \n\max_\alpha |\langle f,h_\alpha\rangle|.
    \end{equation*}
    We can write, for all $\n\ge N_0$
    \begin{eqnarray*}
       \sum_{k\le \n} |\langle f, C_k(S_\theta,x)\rangle| &\le& \sum_{k<N_0} |\langle f, C_k(S_\theta,x)\rangle|  + \sum_{N_0\le k\le \n} |\langle f, C_k(S_\theta,x)\rangle|  \\
       &\le& N^2_0\max_\alpha |\langle f,h_\alpha\rangle|+ \sum_{k\le \n}k^{\Lambda f +\frac{\epsilon}{2}}
    \end{eqnarray*}
    Using comparison series vs integral, we have 
    \begin{equation*}
        \sum_{k\le \n}k^{\Lambda f +\frac{\epsilon}{2}}\le \int_0^{\n+1} x^{\Lambda f +\frac{\epsilon}{2}}dx\le \frac{(\n+1)^{\Lambda f +\frac{\epsilon}{2}+1}}{\Lambda f +\frac{\epsilon}{2}+1}.
    \end{equation*}
    Then we can find two constants $A,B$ (depending of $\epsilon$ only ) such as, for all $\n\ge N_0$, we have 
    \begin{equation*}
        \sum_{k\le \n} |\langle f, C_k(S_\theta,x)\rangle|\le A+B~ \n^{\Lambda f+1+\frac{\epsilon}{2}}\le (A~\n^{-\Lambda f-1-\frac{\epsilon}{2}}+B)~\n^{\Lambda f+1+\frac{\epsilon}{2}}.
    \end{equation*}
If $\n$ goes to $\infty$, the factor $A~\n^{-\Lambda f-1-\frac{\epsilon}{2}}+B$ is bounded, then its smaller than $\n^{\frac{\epsilon}{2}}$ for $\n$ large enought. Then, by increasing $N_0$, we show that: for all $\epsilon$, there is $N_0$ such that, for all $\n\ge N_0$ we have 
    \begin{equation*}
          \sum_{k\le \n} |\langle f, C_k(S_\theta,x)\rangle|\ge \n^{\Lambda f +1 +\epsilon}.
    \end{equation*}
As before, the bound is uniform on $I^*(S_\theta)$.
\end{proof}

\paragraph{Lower bound:}
Now, we need to establish a lower bound for the second statement of proposition \ref{SB}.
\begin{proof} 
First, we exhibit intervals of discrete times on which the sum of $|\langle f,C_k(S_\theta)\rangle|$ is large enough. To do this, we use large excursions (see proposition \ref{LDC}) and the fact that the diffusion is bounded around such excursions (see corollary \ref{upbound2}). The trajectory goes far and stays far enough. In a second, we prove that these intervals can be used to minorate $\frac{\log \sum_{k\le \n}C_k(S_\theta)}{\log \n}$ where $\n$ belongs to some intervals of discrete times. Finally, we show that the union of these intervals contains $N_1+\N$ for $N_1$ large enought.\\

$1)$ For all $p\ge0$, we can write
 \begin{equation*}
\sum_{k\le \n_n+p} |\langle f,C_k(S_\theta)\rangle| \ge \sum_{0\le k\le p} |\langle f,C_{\n_n+k}(S_\theta)\rangle|.
 \end{equation*}
 Using the fact that $C_{\n}(S_\theta,x)$ is an additive cocycle for the interval exchange $T_{S_\theta}$, we have:
 \begin{equation*}
|\langle f,C_{\n_n+k}(S_\theta)\rangle|\ge |\langle f,C_{\n_n}(S_\theta)\rangle|-|\langle f,C_{k}(S_\theta,T_{S_\theta}^{\n_n}x)\rangle|.
 \end{equation*}
Of course, if $x$ has an infinite future orbit, then it is also true for $T_{S_\theta}^{\n_n}x$. We can apply proposition \ref{upbound2} for an $\epsilon$ such that  $0<2\epsilon< \text{min}\{\beta,\Lambda f\}$, we can find $p_\epsilon$ which does not depend on $x,n$ and such that, for all $p\ge p_\epsilon$
\begin{equation*}
\sum_{0\le k\le p}  |\langle f,C_{k}(S_\theta,T_{S_\theta}^{\n_n}x)\rangle| \le \exp\left ((\Lambda f+1+\epsilon)\log p \right).    
\end{equation*}
Using the lower bound of proposition \ref{LDC} and the lower bound on $\n_n$, there is $n_\epsilon$ such that, for all $n\ge n_\epsilon$: 
\begin{eqnarray*}
    \sum_{k\le \n_n+p} |\langle f,C_k(S_\theta,x)\rangle| &\ge& p \exp\left ((\Lambda f-\epsilon)\n_n \right) - \exp\left ((\Lambda f+1+\epsilon)\log p \right)\\
     &\ge& p \exp\left ((\Lambda f-\epsilon)(\beta-\epsilon) n \right) - \exp\left ((\Lambda f+1+\epsilon)\log p \right)\\
      &\ge& p\left ( \exp\left ((\Lambda f-\epsilon)(\beta-\epsilon) n \right) - \exp\left ((\Lambda f+\epsilon)\log p \right)\right).
\end{eqnarray*}
If $p \le 2^{-\frac{1}{\Lambda f+\epsilon}}\exp\left(\frac{\Lambda f-\epsilon}{\Lambda f+\epsilon}(\beta-\epsilon) n \right)$, we then have 
\begin{equation*}
    \exp\left ((\Lambda f-\epsilon)(\beta-\epsilon) n \right) - \exp\left ((\Lambda f+\epsilon)\log p \right)\ge \frac{\exp\left ((\Lambda f-\epsilon)(\beta-\epsilon) n \right) }{2} 
\end{equation*}
Let $p_n$ be the integer part of $\exp(-\frac{\log(2)}{\Lambda f+\epsilon})\exp\left(\frac{\Lambda f-\epsilon}{\Lambda f+\epsilon}(\beta-\epsilon) n \right)$. As $p_n$ goes to $\infty$ when $n$ grows up, there is $n_\epsilon'\ge n_\epsilon$ and $A_2>0$ such that if $n\ge n'_\epsilon$, we have $p_n\ge n_\epsilon$. Then, we have:

\begin{eqnarray*}
      \sum_{k\le \n_n+p_n} |\langle f,C_k(S_\theta)\rangle|  &\ge&  \frac{p_n}{2}\exp \left((\Lambda f- \epsilon ) (\beta-\epsilon) n \right)\\
                &\ge&  \left(2^{-\frac{1}{\Lambda f+\epsilon}} \exp \left( \frac{\Lambda f - \epsilon}{\Lambda f + \epsilon } (\beta-\epsilon) n \right)-1 \right)
    \exp \left(  (\Lambda f- \epsilon ) (\beta-\epsilon) n \right)\\
&\ge& \left(2^{-\frac{1}{\Lambda f+\epsilon}}- \exp \left( -\frac{\Lambda f - \epsilon}{\Lambda f + \epsilon } (\beta-\epsilon) n \right)\right) \exp \left( \left( \Lambda f+\frac{\Lambda f - \epsilon}{\Lambda f + \epsilon}- \epsilon \right) (\beta-\epsilon) n \right).\\
&\ge&A_2 \exp \left( \left( \Lambda f+\frac{\Lambda f - \epsilon}{\Lambda f + \epsilon}- \epsilon \right) (\beta-\epsilon) n \right) .
\end{eqnarray*}
Finally obtain
\begin{eqnarray*}
    \log \sum_{k\le \n_n+p_n} |\langle f,C_k(S_\theta)\rangle| ) &\ge& \log \sum_{\n_n< k\le \n_n+p_n} |\langle f,C_k(S_\theta)\rangle| )\\
&\ge& \left( \Lambda f+\frac{\Lambda f - \epsilon }{\Lambda f + \epsilon }- \epsilon \right) (\beta-\epsilon) n  + \log A_2.
\end{eqnarray*}

$2)$ Using the expression of $p_n$, there exist constants $A_3$ independent of $\epsilon,n$ such that
\begin{equation*}
p_{n}\le  A_3 \exp\left(\frac{\Lambda f - \epsilon }{\Lambda f + \epsilon }(\beta-\epsilon)n\right) \le A_3\exp((\beta+\epsilon)n).
\end{equation*}
If $n$ is large enough, we have obviously
\begin{equation*}
    \frac{\log(A_3 + 1)}{n+1}\le \epsilon,~~~\text{and}~~~\frac{\beta-\epsilon}{n+1}\le \epsilon.
\end{equation*}
And then 
\begin{eqnarray*}
    \frac{(\beta-\epsilon) n}{\log( \n_{n+1}+p_{n+1})}&\ge& \frac{(\beta-\epsilon) n}{\log(A_3 \exp((\beta-\epsilon) (n+1))+\n_{n+1})}\\
    &\ge& \frac{(\beta-\epsilon) n}{\log(A_3 + 1)+ (\beta+\epsilon)(n+1)}\\
    &\ge& \frac{\beta-2\epsilon}{2\epsilon +  \beta}.
\end{eqnarray*}
Then, for all $\n$ such as $ \n_{n}+p_n \le \n$ and $\n \le \n_{n+1}+p_{n+1}$, we have:
\begin{eqnarray*}
     \frac{\log \sum_{k\le \n} |\langle f,C_k(S_\theta)\rangle|}{\log \n}
    &\ge&  \frac{\log \sum_{k\le \n_n+p_n} |\langle f,C_k(S_\theta)\rangle|}{\log( \n_{n+1}+p_{n+1})}\\
    &\ge&\frac{n(\beta-\epsilon)}{\log(\n_{n+1}+p_{n+1})} \left( \Lambda f+\frac{\Lambda f - \epsilon}{\Lambda f + \epsilon }- \epsilon \right)  +\frac{\log A_2}{\log(\n_{n+1}+p_{n+1})} \\
    &\ge&   \frac{\beta -2\epsilon}{2\epsilon +  \beta}  \left( \Lambda(f)+\frac{\Lambda f - \epsilon }{\Lambda f + \epsilon }- \epsilon \right)  + \frac{\log A_2}{\log(\n_{n+1}+p_{n+1})}.
\end{eqnarray*}
As it goes to zeros when $n$ is large, the last term is useless; we can find $n_\epsilon''\ge n_\epsilon'$ such that, for all $n\ge n_\epsilon''$ and for all $\n$ that satisfies 
\begin{equation*}
    \n_{n}+p_{n} \le \n ~\text{and}~ \n \le \n_{n+1}+p_{n+1},
\end{equation*} 
we have 
\begin{equation*}
      \frac{\log \sum_{k\le \n} |\langle f,C_k(S_\theta)\rangle|}{\log \n} \ge  \frac{\beta-2\epsilon}{2\epsilon + \beta}  \left( \Lambda f+\frac{\Lambda f - \epsilon }{\Lambda f + \epsilon }- 2\epsilon \right) .
\end{equation*}
$3)$ It remains to prove that the bound is true for all $\n$ big enough. Let $\n\ge N_\epsilon := \n_{n_\epsilon''}+p_{n_\epsilon''}$, and $n_0$ be the smaller integer bigger than $n_\epsilon'' +1$ such that $\n \le \n_{n_0}+p_{n_0}$. We must have $\n \ge  \n_{n_0 -1}+p_{n_0 -1}$, and then the last inequality is true for this $\n$. Then it's true  for all $\n \ge N_\epsilon$, and we can conclude that 
\begin{equation*}
    \liminf_\n  \frac{\log \sum_{k\le \n} |\langle f,C_k(S_\theta)\rangle|}{\log \n}  \ge  \frac{\beta  
    -2\epsilon}{2\epsilon +  \beta}  \left( \Lambda f+\frac{\Lambda f - 2\epsilon }{\Lambda f + \epsilon }- 2\epsilon \right).
\end{equation*}
To conclude, the bound is valid for all $\epsilon > 0$ small enough. Finally, when $\epsilon$ goes to zero, we obtain the claim:
\begin{equation*}
    \liminf_\n  \frac{\log \sum_{k\le \n} |\langle f,C_k(S_\theta)\rangle|}{\log \n}  \ge  \Lambda f + 1.
\end{equation*}
Because the LHS tends to $\Lambda f + 1$ when $\epsilon$ goes to $0$.
\end{proof}

\paragraph{End of the proof of theorem \ref{SB}}
We achieve the proof of the theorem \ref{SB}.
\begin{proof}
    Let $\theta$ such as $S'_{\frac{\pi}{2}-\theta}$ is generic for both theorems \ref{CE1} and \ref{CE2}. By using lemma \ref{upbound2}, the lower bound and remark \ref{rem_cocycle} we obtain 
\begin{equation*}
    \lim_{\n\rightarrow \infty }\frac{\log\sum_{k\le \n}\|\langle f,C_k(S'_\theta)\|}{\log(\n)}=\Lambda f+1
\end{equation*}
Then, by using lemma \ref{lem_covering_gathering} we can conclude that 
\begin{equation*}
    \lim_{T\rightarrow \infty }\frac{\log\int_0^T d_{\tilde{S}}(\tilde{x},\tilde{\phi}_t^\theta(\tilde{x})}{\log(T)}=\Lambda f+1.
\end{equation*}
Where $\tilde{\phi}_t(\tilde{x})$ is the flow on $\tilde{S}$ in the direction $\theta$. Then, for all $\tilde{x}$ and almost all $\theta$, the statement of theorem \ref{SB} is true.
\end{proof}

\begin{rem}
    To be correct, we do not really prove the theorem for $r_\theta\cdot \tilde{S}'$ but for $\tilde{S}'_\theta$. But the two differ by the action of the Teichmüller flow, and this does not affect the statement of the theorem.
\end{rem}

Using this theorem and the previous works on billiard \cite{delecroix2014diffusion}, we can deduce a similar statement for the wind tree model. This step is now classical in the theory of billiards.

\appendix
\appendix
\section{Proof of lemma \ref{ie}}
An intervals exchange of $n$ intervals can be defined by data $(\mathcal{A},\pi^t,\pi^b,\lambda)$, where  $\pi^{\cdot}: \mathcal{A}\to \llbracket 1,n\rrbracket$. Assume that the data are irreducible (see \cite{viana2008dynamics}), and we denote $(T_\lambda,I_\lambda)$ the corresponding intervals exchange. As before, let $I^*_\lambda$ be the points in $I_\lambda$ that have an infinite future orbit, i.e., the points that never hit a singularity. We denote by $\iota_\lambda(x)$ the sequence of $(\alpha_n)$ such that $T^n_\lambda(x)\in I^t_{\alpha_n}$, it is well defined on $I^*_\lambda$. It is easy to see that this function is continuous for the product topology, $\iota_\lambda\in \mathcal{A}^\N$.
We denote $\Sigma_\lambda$ the closure of $\iota_\lambda(I^*_\lambda)$ in $\mathcal{A}^\N$.
\begin{lem}
The interval exchange is conjugated to the shift on $\Sigma_\lambda$. We have the following diagram, where $h$ is an almost everywhere homeomorphism:
\begin{equation*}
\xymatrix{
\Sigma_\lambda \ar[r]^\sigma \ar[d]_h & \Sigma_\lambda \ar[d]^h\\ 
I_\lambda \ar[r]^T &I_\lambda}
\end{equation*}
\end{lem}
By using this lemma, we can prove the first statement of lemma \ref{ie}.
\begin{proof}
If $T_\lambda$ is a minimal interval exchange, then for all $x\in I^*_\lambda$, the future orbit of $x$ is dense in $I_\lambda$. And then the function
\begin{equation*}
    m_\alpha(x) = \min \{ n, T^n_\lambda x \in I_{\lambda,\alpha} \},
\end{equation*}
is well defined on $I^*_\lambda$. It is easy to check that the function
\begin{equation*}
    \tilde{m}_\alpha(u) = \min \{ n, \alpha_n =\alpha \}
\end{equation*}
is a continuous function on $\Sigma_\lambda$ and also $m_\alpha(x)=\tilde{m}_\alpha(\iota_\lambda(x))$. So, as $\Sigma_\lambda$ is compact, then $\tilde{m}_\alpha$ is bounded on $\Sigma_\lambda$, and then $\max_\alpha m_\alpha$ is bounded on $I^*_\lambda$.
\end{proof}
Now we prove the second part; more precisely, we will prove the following lemma:
\begin{lem}
For all intervals exchange $T_{\lambda_0}$ with no connection, for all $m$, there exists a neighborhood $V$ of $\lambda_0$ such that
\begin{equation*}
    \iota_\lambda(0)_n= \iota_{\lambda_0}(0)_m~~~~~~\forall n<m ~,~\forall \lambda \in V
\end{equation*}
\end{lem}
\begin{proof}
We will prove this result by induction on $n$, assume that's true for $m$, and let $V$ be the neighborhood of $\lambda$. We need to introduce some notations, i.e., let $\alpha_n=\iota_{\lambda_0}(0)_n$ and $\gamma=\alpha_{m+1}$. By assumption, we can write:
\begin{equation*}
    T_\lambda^{m+1}(y)= \sum_{0\le k \le m} \delta_{\iota_\lambda(y)}(\lambda)=\sum_{0\le k \le m} \delta_{\alpha_k}(\lambda) ~~~~\forall y \in ]0,u[.
\end{equation*}
If $\phi(\lambda)= T_\lambda^{m+1}(0)=\sum_{0\le k \le N} \delta_{\alpha_k}(\lambda)$, where $\delta_{\alpha_k}(\lambda)$ is the translation from $I^t_{\lambda,\alpha}$ to $I^b_{\lambda,\alpha}$. It is defined by the following formula:
\begin{equation*}
    \delta_{\alpha}(\lambda)= \underset{\pi^t(\beta) \ge \pi^t(\alpha)}{\sum} \lambda_\beta - \underset{\pi^b(\beta) \ge \pi^b(\alpha)}{\sum} \lambda_\beta.
\end{equation*}
And then $\phi$ is continuous in $\lambda$. If we denote $\epsilon = \frac{d(\phi(\lambda_0),(I_{\lambda_0,\gamma}^t)^c)}{3}$, it exists $V'$ a neighborhood of $\lambda_0$ in $V$ such that
\begin{equation*}
    |\phi(\lambda)-\phi(\lambda_0)|\le \epsilon~~~~\forall \lambda \in V'.
\end{equation*}
Let $u^+_\alpha(\lambda),u^-_\alpha(\lambda)$ such that $I^t_{\lambda,\alpha}=]u^-_\alpha(\lambda),u^+_\alpha(\lambda)[$, we can write
\begin{equation*}
    u^-_\alpha(\lambda)=\underset{\pi^t(\beta) < \pi^t(\alpha)}{\sum} \lambda_\beta,~~~\text{and}~~~u^+_\alpha(\lambda)=u^-_\alpha(\lambda)+\lambda_\alpha.
\end{equation*}
The two functions are continuous, and then we can find an open neighborhood $V''$ such that $\forall \lambda \in V''$
\begin{equation*}
u^-_\gamma(\lambda)   \le  u^-_\gamma(\lambda_0)+\epsilon  < u^+_\gamma(\lambda_0) -\epsilon \le   u^+_\gamma(\lambda).
\end{equation*}
And then 
\begin{equation*}
T_{\lambda}^{m+1}(0) ~\in ~]  T_{\lambda_0}^{m+1}(0) - \epsilon,  T_{\lambda_0}^{m+1}(0)+\epsilon [ ~\subset~ ]u^-_\gamma(\lambda_0)+\epsilon  , u^+_\gamma(\lambda_0) -\epsilon [ ~\subset~ ] u^-_\gamma(\lambda) , u^+_\gamma(\lambda)[
\end{equation*}
and then $\iota_\lambda(0)_{m+1}=\iota_{\lambda_0}(0)_{m+1}$ for all $V''$
\end{proof}
\bibliography{biblio}
\bibliographystyle{alpha}
\end{document}